\documentclass[11pt]{amsart}

\usepackage[a4paper,hmargin=3.5cm,vmargin=4cm]{geometry}
\usepackage{amsfonts,amssymb,amscd,amstext,verbatim}

\usepackage{fancyhdr}
\usepackage{times}

\usepackage{enumerate}
\usepackage{titlesec}
\usepackage{mathrsfs}

\titleformat{\section}
{\filcenter\bfseries\large} {\thesection{.}}{0.2cm}{}
\titleformat{\subsection}[runin]
{\bfseries} {\thesubsection{.}}{0.15cm}{}[.]
\titleformat{\subsubsection}[runin]
{\em}{\thesubsubsection{.}}{0.15cm}{}[.]

\usepackage[up,bf]{caption}

\newtheorem{theorem}{Theorem}[section]
\newtheorem{proposition}[theorem]{Proposition}

\theoremstyle{definition}
\newtheorem{definition}[theorem]{Definition}
\newtheorem{remark}[theorem]{Remark}

\newtheorem{problem}[theorem]{Problem}
\newtheorem{example}[theorem]{Example}

\numberwithin{equation}{section}
\numberwithin{figure}{section}

\newcommand\Cscr{\mathscr{C}}

\newcommand\Oscr{\mathscr{O}}

\newcommand\B{\mathbb{B}}
\newcommand\C{\mathbb{C}}

\newcommand\N{\mathbb{N}}

\newcommand\R{\mathbb{R}}

\newcommand\igot{\mathfrak{i}}

\renewcommand\igot{\mathfrak{i}}

\renewcommand\imath{\igot}

\newcommand\wh{\widehat}

\newcommand\dist{\mathrm{dist}}

\def\dist{\mathrm{dist}}

\begin{document}

\fancyhead[LO]{Surjective holomorphic maps onto Oka manifolds}
\fancyhead[RE]{F.\ Forstneri\v c} 
\fancyhead[RO,LE]{\thepage}

\thispagestyle{empty}

\vspace*{1cm}
\begin{center}
{\bf\LARGE Surjective holomorphic maps onto Oka manifolds}

\vspace*{0.5cm}

{\large\bf  Franc Forstneri\v c} 
\end{center}


\vspace*{1cm}

\begin{quote}
{\small
\noindent {\bf Abstract}\hspace*{0.1cm}
Let $X$ be a connected Oka manifold, and let $S$ be a Stein manifold with $\dim S \ge \dim X$.
We show that every continuous map $S\to X$ is homotopic to a surjective strongly dominating 
holomorphic map $S\to X$. We also find strongly dominating algebraic morphisms from the affine
$n$-space onto any compact $n$-dimensional algebraically subelliptic manifold. 
Motivated by these results, we propose a new holomorphic flexibility property of complex manifolds, the 
{\em basic Oka property with surjectivity}, which could potentially provide another characterization
of the class of Oka manifolds.  

\vspace*{0.2cm}

\noindent{\bf Keywords}\hspace*{0.1cm} Stein manifold, Oka manifold, holomorphic map, algebraic map

\vspace*{0.1cm}


\noindent{\bf MSC (2010):}\hspace*{0.1cm}}  32E10, 32H02, 32Q45, 14A10 
\end{quote}


\section{Introduction} 
\label{sec:intro}

A complex manifold $X$ is said to be an {\em Oka manifold}
if every holomorphic map $U\to X$ from an open convex set $U$ in a complex Euclidean space 
$\C^N$ can be approximated uniformly on compacts in $U$ by holomorphic maps $\C^N\to X$.
This {\em convex approximation property} (CAP) of $X$, which was first introduced in \cite{Forstneric2006AM},
implies that maps  from any Stein manifold $S$ to $X$ satisfy the parametric Oka principle
with approximation and interpolation (see \cite[Theorem 5.4.4]{Forstneric2011};  it suffices
to verify CAP for the integer $N=\dim S+\dim X$).  In particular, every continuous
map $f\colon S\to X$ from a Stein manifold $S$ to an Oka manifold $X$ 
is homotopic to a holomorphic map $F\colon S\to X$, and $F$ can chosen to approximate $f$ 
on a compact $\Oscr(S)$-convex set $K\subset S$ provided that $f$ is  holomorphic on a neighborhood of $K$.  
For the theory of Oka manifolds, we refer to the monograph \cite{Forstneric2011} and the 
surveys \cite{Forstneric2013,ForstnericLarusson2011,Kutzschebauch2014};
for the theory of Stein manifolds, see \cite{GunningRossi,HormanderComplex}.

In this  note, we construct {\em surjective} holomorphic maps from Stein manifolds to Oka manifolds,
and surjective algebraic morphisms of affine algebraic manifolds to certain compact algebraic manifolds.
We say that a (necessarily surjective) holomorphic map $F\colon S\to X$ is {\em strongly dominating}
if for every point $x\in X$ there exists a point $p\in S$ such that $F(p)=x$ and 
$dF_p\colon T_p S\to T_x X$ is surjective.  Equivalently, $F(S\setminus\mathrm{br}_F)=X$
where $\mathrm{br}_F \subset S$ is the branch locus of $F$. 

 
%
%
\begin{theorem} \label{th:main}
Let $X$ be a connected Oka manifold. If $S$ is a Stein manifold and $\dim S\ge \dim X$ then 
every continuous map $f\colon S\to X$ is homotopic to a strongly dominating (surjective) holomorphic map $F\colon S\to X$. 
In particular, there exists a strongly dominating holomorphic map $F\colon \C^n\to X$ for $n=\dim X$.
\end{theorem}

Theorem \ref{th:main}  answers a question that arose in author's discussion with J\"org Winkelmann
(see the Acknowledgement). The result also holds, with the same proof, if $S$ is a reduced Stein space.
A similar result in the algebraic category is given by Theorem \ref{th:algebraic}.

A version of Theorem \ref{th:main}, with $X\subset \C^n$ a non-autonomous basin of 
a sequence of attracting automorphisms with uniform bounds, is due to Forn\ae ss and Wold 
(see \cite[Theorem 1.4]{FornaessWold2014}). 
With the exception of surjectivity, the results in their theorem had been know earlier
for maps of Stein manifolds to any Oka manifold; see \cite[Theorem 7.9.1 and Corollary 7.9.3, pp. 324-325]{Forstneric2011}
for the existence of embeddings, while the existence of maps with dense images 
is a direct consequence of the fact that Oka manifolds enjoy the Oka property with interpolation 
(cf.\ \cite[Theorem 5.4.4]{Forstneric2011}).

According to the standard terminology, a holomorphic map $F\colon \C^n\to X$ 
is said to be {\em dominating} at the point $x_0=F(0)\in X$ if the differential $dF_0\colon T_0\C^n \to T_{x_0}X$  is surjective;
if such $F$ exists then $X$ is  {\em dominable at $x_0$}. A complex manifold which is dominable at every point 
is called {\em strongly dominable}. Every Oka manifold is strongly dominable, 
but the converse is not known. For a discussion of this subject, see  e.g.\ \cite{ForstnericLarusson2014}.
Theorem \ref{th:main} furnishes a map $F:\C^n\to X$ such that the family of maps $\{F\circ\phi_a\}_{a\in\C^n}$,
where $\phi_a\colon \C^n\to \C^n$ is the translation $z\mapsto z+a$, dominates at every point of $X$.

On the other hand, we do not know whether every Oka manifold $X$ is the image of a {\em locally biholomorphic} map 
$\C^n\to X$ with $n=\dim X$. A closely related problem is to decide whether locally biholomorphic self-maps of 
$\C^n$ for $n>1$ satisfy the Runge approximation theorem; see \cite[Problem 8.11.3 and Theorem 8.12.4]{Forstneric2011}.

Theorem \ref{th:main} is proved in Section \ref{sec:proofmain}. The proof is based on 
an approximation result for holomorphic maps
from Stein manifolds to Oka manifolds which we formulate in Section \ref{sec:approx1} (see Theorem \ref{th:approx1}).
The approximation takes place on a locally finite sequence of compact sets in a
Stein manifold $S$ which are separated by the level sets of a strongly
plurisubharmonic exhaustion function and satisfy certain holomorphic convexity conditions.
Although Theorem \ref{th:approx1} follows easily from the proof of the Oka principle
with approximation (see e.g.\ \cite[Chapter 5]{Forstneric2011}), this  formulation is useful in 
certain situations like the one considered here, and hence we feel it worthwhile to record it.

Theorem \ref{th:main} is motivated in part by results to the effect that certain complex 
manifolds $S$ are {\em universal sources}, in the sense that they admit a surjective holomorphic map 
$S\to X$ onto every complex manifold of the same dimension. This holds for  a polydisk and a ball in $\C^n$ 
(see Forn{\ae}ss and Stout \cite{FornaessStout1977,FornaessStout1982}; in this case, the map
can be chosen locally biholomorphic and finitely sheeted),
and also for any bounded domain with $\Cscr^2$ boundary in $\C^n$ (see L\o w \cite{Low1983}).
Further results, with emphasis on the case $X=\C^n$, were obtained
by Chen and Wang \cite{ChenWang2015}. In these results, the source manifold is Kobayashi hyperbolic. 
This condition cannot be substantially weakened since a holomorphic map 
is distance decreasing with respect to the Kobayashi pseudometrics on the respective manifolds.
In particular, a manifold with vanishing Kobayashi pseudometric (such as $\C^n$)
does not admit any nonconstant holomorphic
map to a hyperbolic manifold. Furthermore, the existence of a nondegenerate holomorphic map $\C^n\to X$ 
to a connected compact complex manifold $X$ of dimension $n$ implies that $X$ is not of general type 
(see Kodaira \cite{Kodaira1971} and Kobayashi and Ochiai \cite{KobayashiOchiai1975}).
By an extension of the Kobayashi-Ochiai argument, Campana proved that
such $X$ is actually {\em special} \cite[Corollary 8.11]{Campana2004}.
Special manifolds are important in Campana's structure theory of compact K\"ahler manifolds.
Recently, Diverio and Trapani \cite{DiverioTrapani2016} and Wu and Yau \cite{WuYau2016P,WuYau2016IM} 
proved that a compact connected complex manifold $X$, which admits a K\"ahler metric whose 
holomorphic sectional curvature is everywhere nonpositive and is strictly negative at least at one point, 
has positive canonical bundle $K_X$.
(See also Tosatti and Yang \cite{TosattiYang2016} and Nomura \cite{Nomura2016}.)
Hence, such $X$ is projective and of general type, and therefore it does not admit any 
nondegenerate holomorphic map $\C^n\to X$ with $n=\dim X$.

These observations justify the hypothesis in Theorem \ref{th:main} that $X$ be an Oka manifold.

Let us recall a related but weaker holomorphic flexibility property introduced by Gromov \cite{Gromov1989}.
A complex manifold $X$ is said to enjoy the {\em basic Oka property}, BOP, 
if every continuous map $S\to X$ from a Stein manifold $S$ is homotopic to a holomorphic map. 
The only difference with respect to the class of Oka manifolds is that the BOP axiom does 
not include any approximation or interpolation conditions. Thus, every Oka manifold satisfies BOP, 
but the converse fails e.g.\ for contractible hyperbolic manifolds (such as bounded convex domains in $\C^n$). 
The basic Oka property property was studied  by Winkelmann \cite{Winkelmann1993} for maps between 
Riemann surfaces, and by Campana and Winkelmann \cite{CampanaWinkelmann2015} for more general 
complex manifolds. (Their use of the term {\em homotopy principle} is equivalent to BOP.)
Recently, Campana and Winkelmann proved (see \cite{Winkelmann2016}) that a compact projective manifold 
satisfying BOP is special in the sense of  \cite{Campana2004}. We thus have the implications 
\[
     \text{${\rm Oka} \Longrightarrow {\rm BOP} \Longrightarrow  {\rm special}$},
\]
where the second one holds for compact projective manifolds (and is expected to be true for all
compact K\"ahler manifolds).

Concerning the relationship between Oka manifolds and manifolds with BOP, one has the feeling that these 
two classes  are essentially the same after eliminating the obvious counterexamples provided by contractible 
hyperbolic manifolds; the latter may be used as building blocks in manifolds with BOP, but not in Oka manifolds. 
With this in mind, we propose the following Oka property.

%
%
\begin{definition}
A connected complex manifold $X$ satisfies the {\em basic Oka property with surjectivity}, abbreviated BOPS,
if  every continuous map $f:S\to X$ from a Stein manifold $S$ with $\dim S\ge \dim X$ is 
homotopic to a {\em surjective} holomorphic map $F:S\to X$. 
\end{definition}

Theorem \ref{th:main} says that  ${\rm Oka}\Rightarrow{\rm BOPS}$.
Applying the BOPS axiom to a constant map $\C^n\to x_0\in X$ gives the following observation.

\begin{proposition}
A connected complex manifold $X$ satisfying {\rm BOPS} admits a surjective holomorphic map 
$\C^n\to X$ with $n=\dim X$. In particular, the Kobayashi pseudometric of a 
complex manifold satisfying {\rm BOPS} vanishes identically.
\end{proposition}

Since the BOPS axiom eliminates the obvious counterexamples to the (false) implication
BOP $\Rightarrow$ Oka, the following seems a reasonable questions.

%
%
\begin{problem}\label{prob:BOP-Oka}
(a) Assuming that a complex manifold $X$ satisfies BOPS, does it follow that $X$ is an Oka manifold?
That is, do we have the implication ${\rm BOPS} \Rightarrow {\rm  Oka}$?

(b) Do the properties BOP and BOPS coincide in the class of compact (or compact K\"ahler,
or compact projective) manifolds?
\end{problem}

Let us mention another question related to Theorem \ref{th:main}. Let $\B^n$ denote the open ball in $\C^n$. 
It is an open problem whether $\C^n\setminus \overline\B^n$ is an Oka manifold
when $n>1$.

\begin{problem}\label{prob:ball}
Let $n>1$. Does there exist a surjective holomorphic map $\C^n\to \C^n\setminus \overline\B^n$?
\end{problem}

In this connection, we mention that Dixon and Esterle (see \cite[Theorem 8.13, p.\ 182]{DixonEsterle1986})
constructed for every $\epsilon>0$ a finitely sheeted holomorphic map $f\colon \C^2\to \C^2$ 
whose image avoids the closed unit ball $\overline \B^2$ but contains the complement of the ball of 
radius $1+\epsilon$: $\C^2\setminus (1+\epsilon)\overline \B^2 \subset f(\C^2) \subset \C^2\setminus\overline \B^2$.

Theorem \ref{th:main} shows that a negative answer to Problem \ref{prob:ball} would imply that 
$\C^n\setminus \overline\B^n$ fails to be Oka. 
Since $\C^n\setminus \overline\B^n$ is a union of Fatou-Bieberbach domains 
(obtained for example as attracting basins of holomorphic automorphisms of $\C^n$ 
which map the ball $\B^n$ into itself), this would provide an example of a strongly dominable manifold 
which is not Oka. 

The above example is also connected to the open problem whether every Oka manifold is elliptic or subelliptic,
the latter being the main known geometric conditions implying all versions of the Oka property 
(see Gromov \cite{Gromov1989}, Forstneri\v c \cite{Forstneric2002MZ},
and \cite[Definition 5.5.11 (d) and Corollary 5.5.12]{Forstneric2011}).  
The following implications hold for any complex manifold:
\[
	{\rm homogeneous} \Longrightarrow  {\rm elliptic} \Longrightarrow {\rm subelliptic} 
	\Longrightarrow {\rm Oka} \Longrightarrow {\rm strongly\ dominable}.
\]
It was shown by Andrist et al \cite{Andristetal2016} that $\C^n\setminus \overline\B^n$
is not subelliptic when $n\ge 3$. Since $\C^n\setminus \overline\B^n$ is strongly dominable,
at least one of the two right-most implications  cannot be reversed.
Therefore, the question whether $\C^3\setminus\overline\B^3$ is an Oka manifold is of particular interest.

It is natural to look for an analogue of Theorem \ref{th:main} in the algebraic category. 
At this time, we do not have a good notion of an algebraic Oka manifold. 
However, a useful geometric condition on an algebraic manifold $X$, which gives
the approximation of certain holomorphic maps $S\to X$ from affine algebraic manifolds
$S$ by algebraic morphisms $S\to X$, is {\em algebraic subellipticity}; see \cite[Definition 5.5.11 (e)]{Forstneric2011}
or Section \ref{sec:algebraic} below. (We emphasize that all algebraic maps 
in this paper are understood to be morphisms, i.e., without singularities.)
In Section \ref{sec:algebraic} we prove the following result in this direction. 

%
%
\begin{theorem}\label{th:algebraic}
Assume that $X$ is a compact algebraically subelliptic manifold and 
$S$ is an affine algebraic manifold such that $\dim S\ge \dim X$. Then, every
algebraic map $S\to X$ is homotopic (through algebraic maps) 
to a surjective strongly dominating algebraic map $S\to X$. In particular, $X$ admits a 
surjective strongly dominating algebraic map $F\colon \C^n\to X$ with $n=\dim X$.
\end{theorem}

The proof of Theorem \ref{th:algebraic} is a based on Theorem \ref{th:hRunge-algebraic} which
is  taken from \cite{Forstneric2006AJM}. It says  in particular  that,
given an affine algebraic manifold $S$ and  an algebraically subelliptic manifold $X$,
a holomorphic map $S\to X$ that is homotopic to an algebraic map through a family of 
holomorphic maps can be approximated   
by algebraic maps $S\to X$.

\begin{example}
Let $X$ be an algebraic manifold of dimension $n$ which is covered by Zariski open sets that are
biregularly isomorphic to $\C^n$. 
Then $X$ is algebraically subelliptic (see \cite[Definition 6.4.5 and Proposition 6.4.6]{Forstneric2011}). 
Furthermore, the total space $Y$ of any blow-up $Y\to X$ along a closed (not necessarily connected) submanifold of $X$
is also algebraically subelliptic according to L\'arusson and Truong \cite{LarussonTruong2016}. 
If $Y$ is compact, then Theorem \ref{th:algebraic} furnishes a strongly dominating morphisms $\C^n\to Y$. 
This holds for example if $Y$ is obtained by blowing up a projective space or a Grassmanian 
along a compact submanifold. 
\end{example}


\section{Approximation of maps from a Stein manifold to an Oka manifold on a sequence of Stein compacts}  \label{sec:approx1} 

We denote by $\Oscr(S)$ the algebra of all holomorphic functions on a complex manifold $S$,
endowed with the compact-open topology. Recall that a compact set $K$ in $S$ is said to be 
{\em $\Oscr(S)$-convex} if $K=\wh K_{\Oscr(S)}$, where the holomorphic hull of $K$ is defined by
\[
	\wh K_{\Oscr(S)} = \bigl\{p\in S: |f(p)|\le \sup_K|f|\ \ \forall f\in\Oscr(S)\bigr\}.
\]
In this section, we prove the following approximation result. In the next section, we will apply 
it to prove Theorem \ref{th:main}. 


\begin{theorem} \label{th:approx1}
Let $S$ be a reduced Stein space and $(K_j)_{j=1}^\infty$ be a sequence
of compact pairwise disjoint subsets of $S$ satisfying the following properties:
\begin{itemize}
\item[\rm (a)] Every compact set in $S$ intersects at most finitely many of the sets $K_j$.
\vspace{1mm}
\item[\rm (b)] The union $\cup_{j=1}^k K_j$ is $\Oscr(S)$-convex for each $k\in\N$.
\vspace{1mm}
\item[\rm (c)]  Set $K=\cup_{j=1}^\infty K_j$. 
There exist a strongly plurisubharmonic exhaustion function $\rho\colon S\to\R_+=[0,+\infty)$
and an increasing sequence $0< a_1<a_2<\ldots$ with $\lim_{j\to\infty}a_j=+\infty$ 
such that for every $j\in \N$ we have $K\cap\{\rho=a_j\}=\emptyset$ and 
\begin{equation}\label{eq:M-j}
	  \text{the compact set $M_j := \{\rho\le a_j\} \cup (K\cap \{\rho\le a_{j+1}\})$ is $\Oscr(S)$-convex.} 
\end{equation}
\end{itemize}
Let $X$ be an Oka manifold, and let $f\colon S\to X$ be a continuous map which is holomorphic
on a neighborhood of the set $K=\cup_{j=1}^\infty K_j$. Let $\dist$ be a distance function
on $X$ inducing the manifold topology. Given a sequence $\epsilon_j>0$ $(j\in\N)$, there
exists a holomorphic map $F\colon S\to X$, homotopic to $f$ by a family of maps 
$F_t\colon S\to X$ $(t\in[0,1])$ that are holomorphic on a neighborhood of $K$, such that
\begin{equation}\label{eq:estimate1}
	\sup_{p \in K_j}\dist(f(p),F(p)) <\epsilon_j\quad \text{for all}\ j=1,2,\ldots.
\end{equation}
Furthermore, given a discrete sequence of points $(p_j)_{j\in\N} \subset K$
and integers $k_j\in \N$, we can choose $F$ to agree with $f$ to order $k_j$ at $p_j$.
\end{theorem}

\begin{proof}
We may assume that $\dist$ is a complete metric on $X$ and that $\sum_j \epsilon_j<\infty$. 
Let $(a_j)_{j\in \N}$ be the sequence of real numbers in condition (c). Set
\[
	S_j:=\{p\in S: \rho(p)\le a_j\},\quad  A_{j}:=\{p\in S: a_j\le \rho(p) \le a_{j+1}\},\quad j\in\N.
\]
Note that $S_j$ is compact $\Oscr(S)$-convex, and we have 
\[
	S_{j+1}=S_j \cup A_j \ \ \text{and} \ \  M_j = S_j\cup  (K\cap A_j)\ \ 
	\text{for every}\  j=1,2,\ldots.
\]
(Recall that $K=\cup_{j=1}^\infty K_j$.) For consistency of notation we also set 
\[	
	S_0=\emptyset,\quad M_0:=K\cap S_1,\quad F_0=f. 
\]
By hypothesis (c), we have that $K\cap bS_j=\emptyset$ for all $j\in \N$. 
Furthermore, condition  (a) in the theorem implies that each set $S_j$ contains at most finitely 
many of the sets $K_i$. Set
\begin{equation}\label{eq:eta-j}
	\eta_j:=\min\{\epsilon_i : K_i \subset S_j\}>0, \quad j=1,2,\ldots.
\end{equation}

To prove the theorem, we shall construct sequences of continuous maps $F_j\colon S\to X$, homotopies 
$F_{j,t}\colon S\to X$ $(t\in [0,1])$, and numbers $b_j,c_j>0$ satisfying the following conditions for every $j\in\N$:
\begin{itemize}
\item[\rm (i$_j$)]   $a_j<b_j <c_j <a_{j+1}$ and $K\cap A_j\subset \{c_j <\rho <a_{j+1}\}$.
\vspace{1mm}
\item[\rm (ii$_j$)]   $F_j$ is holomorphic on $\{\rho<b_j\}$ and $F_j=F_{j-1}$ on $\{\rho\ge c_j\}$.
\vspace{1mm}
\item[\rm (iii$_j$)]  $\dist(F_j(p),F_{j-1}(p)) < 2^{-j}\eta_j$ for every $p\in M_{j-1}$. 
\vspace{1mm}
\item[\rm (iv$_j$)]  $F_{j,0}=F_{j-1}$ and $F_{j,1}=F_j$.
\vspace{1mm}
\item[\rm (v$_j$)]  For every $t\in[0,1]$ the map $F_{j,t}$ is holomorphic on a neighborhood of $M_{j-1}$ and
$F_{j,t}=F_{j-1}$ holds on $\{\rho\ge c_j\}$.
\vspace{1mm}
\item[\rm (vi$_j$)]  $\dist(F_{j,t}(p),F_{j-1}(p)) < 2^{-j}\eta_j$ for every $p\in M_{j-1}$ and $t\in[0,1]$.
\end{itemize}
We could also add a suitable condition on $F_j$ to ensure jet interpolation along a discrete sequence 
$(p_j)\subset K$ (see the last sentence in the theorem). Since this interpolation
is a trivial addition in what follows, we shall delete it to simplify the exposition.

A sequence of maps and homotopies satisfying these properties can be constructed recursively
by using \cite[Theorem 5.4.4]{Forstneric2011} at every step; we offer some details.

Assume that maps $F_0,F_1,\ldots,F_j$ and homotopies $F_{1,t},\ldots, F_{j,t}$
with these properties have been found for some $j\in\N$. 
(Recall that $F_0=f$.) In view of property (ii$_j$) the map $F_j$ is 
holomorphic on the set $\{\rho<b_j\}$, and we have
$F_j=F_{j-1}=\cdots=F_0$ on $\{\rho\ge c_j\}$. Since $K\cap A_j\subset \{c_j <\rho <a_{j+1}\}$
by property (i$_j$), it follows that $F_j$ is holomorphic on a neighborhood of the set $M_j$ \eqref{eq:M-j}. 
Since $M_j$ is $\Oscr(S)$-convex,  we can apply \cite[Theorem 5.4.4]{Forstneric2011} fo find
a number $c_{j+1}>a_{j+1}$ close to $a_{j+1}$, a holomorphic map 
$F_{j+1}\colon \{\rho<c_{j+1}\}\to X$ satisfying property (iii$_{j+1}$), and  
a homotopy of maps $F_{j+1,t}\colon \{\rho<c_j\}\to X$ $(t\in [0,1])$ satisfying properties 
(iv$_{j+1}$) and (vi$_{j+1}$). 
It remains to extend this homotopy to all of $S$ such that  condition (v$_{j+1}$) holds as well. 
This is accomplished by using a cut-off function in the parameter of
the homotopy.  Explicitly, pick a number $b_{j+1}$ such that $a_{j+1}<b_{j+1} <c_{j+1}$,
and let $\chi\colon S\to [0,1]$ be a continuous function which equals $1$ on the set $\{\rho\le b_{j+1}\}$
and has support contained in $\{\rho < c_{j+1}\}$. The homotopy of continuous maps
\[
	(p,t) \longmapsto F_{j+1, \chi(p)t}(p) \in X,\quad p\in S,\ t\in[0,1]
\]
then agrees with the homotopy $F_{j+1,t}$ on the set $\{\rho\le b_{j+1}\}$ (since $\chi=1$ there), 
and it agrees with the map $F_{j}$ (and hence with $F_0=f$) on $\{\rho \ge c_{j+1}\}$ since $\chi$
vanishes there. This established the condition (v$_{j+1}$)  and completes the induction step.

In view of  (iii$_j$) and the definition of the numbers $\eta_j$ \eqref{eq:eta-j}, 
the sequence $F_j\colon S\to X$  converges uniformly on compacts in $S$ to a 
holomorphic map $F=\lim_{j\to\infty} F_j\colon S\to X$ satisfying  the estimates \eqref{eq:estimate1}. 
Furthermore, conditions (iv$_{j+1}$)--(vi$_{j+1}$) imply that the sequence of homotopies 
$F_{j,t}\colon S\to X$ $(j\in \N)$ can be assembled into a
homotopy $F_t\colon S\to X$ $(t\in [0,1])$ connecting $F_0=f$ to the final 
holomorphic map $F_1=F$ such that $F_t$ is holomorphic on a neighborhood of the set $K$
for every $t\in[0,1]$ and every map $F_t$ in the homotopy satisfies the estimates \eqref{eq:estimate1}. 
This assembling is accomplished by writing $[0,1)=\cup_{j=1}^\infty I_j$, where $I_j=[1-2^{-j+1},1- 2^{-j}]$,
and placing the homotopy $(F_{j,t})_{t\in [0,1]}$ onto the subinterval $I_j\subset [0,1]$ 
by suitably reparametrizing the $t$-variable. 
\end{proof}

\section{Construction of surjective holomorphic maps to Oka manifolds}  \label{sec:proofmain} 

\begin{proof}[Proof of Theorem \ref{th:main}]
Let $X$ be a complex manifold of dimension $n$.
Choose a countable family of compact sets $L'_j \subset L_j\subset X$ $(j\in\N)$  satisfying the following conditions:
\begin{itemize}
\item[\rm (i)]  $L'_j\subset\mathring L_j$ for every $j\in\N$.
\vspace{1mm}
\item[\rm (ii)]   $\cup_{j=1}^\infty L'_j=X$.
\vspace{1mm}
\item[\rm (iii)] For every $j\in\N$ there are an open set $V_j\subset X$ containing $L_j$ and a biholomorphic 
map $\psi_j\colon V_j\to \phi_j(V_j)\subset \C^n$ such that $\psi_j(L_j)=\overline \B^n$ is the closed unit ball in $\C^n$.    
\end{itemize}
A compact set $L_j\subset X$ satisfying condition (iii) will be called a (closed) {\em ball} in $X$. 
If the manifold $X$ is compact, then we can cover it by a finite family of such balls.

Let $S$ be a Stein manifold of dimension $m=\dim S\ge n$. Choose a smooth strongly plurisubharmonic 
exhaustion function $\rho\colon S\to\R_+=[0,+\infty)$. Pick an increasing  sequence of 
real numbers $a_j>0$ with $\lim_{j\to\infty} a_j=+\infty$. 
For each $j\in\N$ we choose a small $\Oscr(S)$-convex ball $K_j$ in $S$ such that 
\begin{equation}\label{eq:incl}
	K_j\subset \{p\in S\colon a_j < \rho(p) < a_{j+1}\} 
\end{equation}
and 
\begin{equation}\label{eq:holoconvex}
 \text{the compact set $M_j:=K_j\cup \{\rho\le a_j\}$ is $\Oscr(S)$-convex.}
\end{equation}
The last condition can be achieved by taking the balls $K_j$ small enough; here 
is an explanation. By the assumption, there are an open set $U_j\subset S$ containing $K_j$ 
and a biholomorphic coordinate map $\phi_j\colon U_j\to \phi_j(U_j)\subset \C^m$ such that
$\phi_j(K_j)=\overline \B^m\subset\C^m$. In view of \eqref{eq:incl}
we may assume that $\overline U_j \cap \{\rho\le a_j\}=\emptyset$.
Let $p_j:=\phi_j^{-1}(0)\in K_j$ be the center of $K_j$. 
The compact set $\{\rho\le a_j\}\cup  \{p_j\}$ is clearly $\Oscr(S)$-convex, and hence it has a 
basis of compact $\Oscr(S)$-convex neighborhoods. In particular, there is a compact 
neighborhood $T \subset U_j$ of the point $p_j$ such that $T \cup \{\rho\le a_j\}$ is $\Oscr(S)$-convex.
Choose a number $0<r_j<1$ small enough such that $r_j\overline \B^m \subset \phi_j(T)$.
The ball $K'_j:=\phi_j^{-1}(r_j\overline \B^m)$ is then contained in $T$ and is $\Oscr(T)$-convex.
Hence, the set $K'_j\cup \{\rho\le a_j\}$ is $\Oscr(S)$-convex. Replacing $K_j$ 
by $K'_j$ and rescaling the coordinate map $\phi_j$ accordingly
so that it takes this set onto $\overline\B^m$, condition \eqref{eq:holoconvex} is satisfied.

Denote by $\pi\colon \C^m\to \C^n$ the coordinate projection $(z_1,\ldots,z_m)\mapsto (z_1,\ldots,z_n)$.
(Recall that $m\ge n$.)  Then $\pi(\overline \B^m)=\overline\B^n$. Let $U_j\supset K_j$
and $\phi_j\colon U_j\to \phi_j(U_j)\subset \C^m$ be as above. There is an open neighborhood 
$U'_j\subset S$ of $K_j$, with $U'_j\subset U_j$, such that the map 
\[
	f_j=\psi^{-1}\circ \pi \circ\phi_j   \colon U'_j\to V_j\subset X
\]
is a well defined holomorphic submersion satisfying $f_j(K_j)=L_j$ for every $j\in\N$.
 
Since the set $L'_j$ is contained in the interior of $L_j$ and $f_j$ is a submersion, there is a  compact set 
$K'_j$ contained in the interior of $K_j$ such that $L'_j\subset f_j(\mathring K'_j)$.
By Rouch\'e's theorem we can choose $\epsilon_j>0$ small enough such that for every holomorphic map 
$F \colon K_j\to X$ defined on a neighborhood of $K_j$ we have that
\begin{equation}\label{eq:Rouche}
	\sup_{p\in K_j}\dist(f_j(p),F(p))<\epsilon_j  \ \Longrightarrow\  L'_j \subset F(K'_j).
\end{equation}

Let $f\colon S\to X$ be a continuous map. By a homotopic deformation of $f$, supported on a 
contractible neighborhood of the ball $K_j\subset S$ for each $j$, 
we can arrange that $f=f_j$ on a neighborhood of $K_j$ for each $j\in \N$. 
The homotopy is kept fixed outside a somewhat bigger neighborhood of each $K_j$ in $S$,
and these neighborhoods are chosen to have pairwise disjoint closures. 
We denote the new map by the same letter $f$.

Theorem \ref{th:approx1}, applied to the map $f$ and the sequences $K_j$ and $\epsilon_j$,
furnishes a holomorphic map $F\colon S\to X$ that is homotopic to $f$ and satisfies the estimate
\[	
	\sup_{p \in K_j}\dist(f(p),F(p)) <\epsilon_j, \quad  j=1,2,\ldots
\]
(see \eqref{eq:estimate1}). By the choice of $\epsilon_j$ \eqref{eq:Rouche} it follows that 
$L'_j \subset  F(K'_j)$ for each $j\in \N$, and hence 
\[
	F(S)= \cup_{j=1}^\infty F(K'_j) = \cup_{j=1}^\infty  L'_j=X. 
\]
Furthermore, if the numbers $\epsilon_j>0$ are chosen 
small enough, then $F$ has maximal rank equal to $\dim X$ at every point of $K'_j$
(since this holds for the map $f$ on the bigger set $K_j$), and hence
$F$ is strongly dominating. 
\end{proof}

\begin{remark}
The same proof applies if $S$ is a reduced Stein space with $\dim S \ge \dim X$.
In this case, we just pick the balls $K_j$ \eqref{eq:incl}  in the regular locus of $S$.
\end{remark}

\section{Surjective algebraic maps to compact algebraically subelliptic manifolds}\label{sec:algebraic}

In this section we prove Theorem \ref{th:algebraic}. We begin by recalling the relevant notions.

An algebraic manifold $X$ is said to be {\em algebraically subelliptic} if it admits a finite family of algebraic
sprays $s_j\colon E_j\to X$ $(j=1,\ldots,k)$, defined on total spaces $E_j$ of algebraic vector bundles $\pi_j\colon E_j\to X$,
which is {\em dominating} in the sense that for each point $x\in X$ the vector subspaces 
$(ds_j)_{0_x}(E_{j,x}) \subset T_x X$ span the tangent space $T_x X$:
\[
	(ds_1)_{0_x}(E_{1,x}) +  \cdots + (ds_k)_{0_x}(E_{k,x}) = T_x X\quad \forall x\in X.
\]
See \cite[Definition 2.1]{Forstneric2006AJM} or \cite[Definition 5.5.11 (e)]{Forstneric2011} for the details. 
Here, $X$ could be a projective (or quasi-projective) algebraic manifold, although the same theory applies 
to more general algebraic manifolds. By an {\em algebraic map}, we always mean an algebraic morphism
without singularities.

The following result is \cite[Theorem 3.1]{Forstneric2006AJM}; see also \cite[Theorem 7.10.1]{Forstneric2011}.

%
%
%
\begin{theorem} \label{th:hRunge-algebraic}
Assume that $S$ is an affine algebraic manifold and $X$ is an algebraically subelliptic manifold.
Given an algebraic map $f\colon S \to X$, a compact $\Oscr(S)$-convex subset 
$K$ of $S$, an open set $U \subset S$ containing $K$, and a homotopy $f_t\colon U\to X$ of 
holomorphic maps $(t\in [0,1])$ with $f_0=f|_{U}$, there exists for every  $\epsilon >0$
an algebraic map $F\colon S\times \C\to X$ such that 
\[
    F(\cdotp,0)=f \quad\text{and}\quad  \sup_{p\in K,\, t\in [0,1]} \dist\left(F(p,t),f_t(p)\right)<\epsilon.
\]
\end{theorem}


%
%
\begin{proof}[Proof of Theorem \ref{th:algebraic}]
The proof uses Theorem \ref{th:hRunge-algebraic} and is similar to that of Theorem \ref{th:main}.
The main difference is that the initial map $f\colon S\to X$ must be algebraic.
For the sake of simplicity, we present the details only in the special case when $S=\C^n$ with $n=\dim X$.

Fix a point $x_0\in X$ and let $f\colon \C^n\to X$ be the constant map $f(z)=x_0 \in X$.

Since $X$ is compact, there is finite family of pairs of compact sets $L'_j \subset L_j\subset X$ $(j=1,\ldots,\ell)$
satisfying properties (i)--(iii) stated at the beginning of proof of Theorem \ref{th:main} 
(see Section \ref{sec:proofmain}). In particular, each set $L_j$ is a ball in a suitable local coordinate,
and we have that $\cup_{j=1}^\ell L'_j=X$.

Let $n=\dim X$. Choose pairwise disjoint closed balls $K_1,\ldots,K_\ell$ in $\C^n$ whose union
$K:=\cup_{j=1}^\ell K_j$ is polynomially convex. Let $p_j\in K_j$ denote the center of $K_j$.
For each $j=1,\ldots,\ell$ there are an open ball $U_j\subset\C^n$ containing $K_j$ and a biholomorphic map
$g_j\colon U_j\to g_j(U_j)\subset X$ such that $g_j(K_j)=L_j$.
We may assume that the sets  $U_1,\ldots,U_\ell$ are pairwise disjoint.
By using a contraction of $K_j$ and $L_j$ to their respective centers, and after shrinking the neighborhoods
$U_j\supset K_j$ if necessary, we can find homotopies of holomorphic maps $f_{j,t}\colon U_j\to X$
$(t\in [0,1],\ j=1,\ldots,\ell)$ such that 
\[
	f_{j,0}= f|_{U_j}\quad \text{and}\quad f_{j,1} = g_j\quad\text{for all}\ j=1,\ldots,\ell.
\]
Set $U=\cup_{j=1}^\ell U_j$ and denote by $f_t\colon U\to X$ the holomorphic map whose restriction
to $U_j$ agrees with $f_{j,t}$ for each $ j=1,\ldots,\ell$. Then $f_0=f|_{U}$ is the constant map
$U\to x_0$. 

Applying Theorem \ref{th:hRunge-algebraic} to the source manifold $S=\C^n$, 
the constant (algebraic) map $f\colon S\to x_0\in X$, and the homotopy $\{f_t\}_{t\in[0,1]}$ 
furnishes an algebraic map $F\colon \C^n\to X$ whose restriction to $K_j$ approximates the map $g_j$
for each $j=1,\ldots,\ell$. Assuming that the approximation is close enough, 
we see as in the proof of Theorem \ref{th:main} that $F(\C^n)=X$ and that $F$ can be chosen
to be strongly dominating.
\end{proof}

\begin{remark}\label{rem:algflexible}
A major source of examples of algebraically elliptic manifolds
are the {\em algebraically flexible} manifods; see e.g.\ \cite[Definition 12]{Kutzschebauch2014}.
An algebraic manifold $X$ is said to be algebraically flexible of it admits finitely many algebraic vector fields
$V_1,\ldots,V_N$ with complete algebraic flows $\phi_{j,t}$ $(t\in\C,\ j=1,\ldots,N)$,
such that the vectors  $V_1(x),\ldots,V_N(x)$ span the tangent space $T_x X$ at every point $x\in X$. 
Note that every $(\phi_{j,t})_{t\in \C}$ is a unipotent 1-parameter group of algebraic automorphisms of $X$.
The composition of the flows $\phi_{1,t_1}\circ \cdots \circ \phi_{N,t_N}$ is a 
dominating algebraic spray $X\times \C^N\to X$, and hence such $X$ is algebraically elliptic.

For a survey of this subject, we refer to Kutzschebauch's paper \cite{Kutzschebauch2014}.
\end{remark}


\subsection*{Acknowledgements}
The author is supported in part by the grants P1-0291 and  J1-7256 from 
ARRS, Republic of Slovenia.  This work was done during my visit at the Center for Advanced Study 
in Oslo, and I wish to thank this institution for the invitation, partial support and excellent working condition.

I thank J\"org Winkelmann for having asked the question that is answered (in a more precise form) 
by  Theorems \ref{th:main} and \ref{th:algebraic}, and Fr\'ed\'eric Campana for discussions concerning
the relationship between the basic Oka property and specialness of compact complex manifolds.
These communications took place at the conference {\em Frontiers in Elliptic Holomorphic Geometry} 
in Jevnaker, Norway in October 2016. 
I thank Finnur L\'arusson for helpful suggestions concerning the terminology and the precise statements of Theorems \ref{th:main} and \ref{th:algebraic}. Finally, I thank Simone Diverio for references to the recent 
developments on K\"ahler manifolds with semi-negative holomorphic sectional curvature, 
and Tyson Ritter for having pointed out the example by Dixon and Esterle
related to Problem  \ref{prob:ball}.


{\bibliographystyle{abbrv} \bibliography{Surjective}}

\begin{thebibliography}{10}

\bibitem{Andristetal2016}
R.~B. Andrist, N.~Shcherbina, and E.~F. Wold.
\newblock The {H}artogs extension theorem for holomorphic vector bundles and
  sprays.
\newblock {\em Ark. Mat.}, 54(2):299--319, 2016.

\bibitem{Campana2004}
F.~Campana.
\newblock Orbifolds, special varieties and classification theory.
\newblock {\em Ann. Inst. Fourier (Grenoble)}, 54(3):499--630, 2004.

\bibitem{CampanaWinkelmann2015}
F.~Campana and J.~Winkelmann.
\newblock On the {$h$}-principle and specialness for complex projective
  manifolds.
\newblock {\em Algebr. Geom.}, 2(3):298--314, 2015.

\bibitem{ChenWang2015}
B.-Y. Chen and X.~Wang.
\newblock Holomorphic maps with large images.
\newblock {\em J. Geom. Anal.}, 25(3):1520--1546, 2015.

\bibitem{DiverioTrapani2016}
S.~Diverio and S.~Trapani.
\newblock Quasi-negative holomorphic sectional curvature and positivity of the
  canonical bundle.
\newblock Preprint arXiv:1606.01381.

\bibitem{DixonEsterle1986}
P.~G. Dixon and J.~Esterle.
\newblock Michael's problem and the {P}oincar\'e-{F}atou-{B}ieberbach
  phenomenon.
\newblock {\em Bull. Amer. Math. Soc. (N.S.)}, 15(2):127--187, 1986.

\bibitem{FornaessStout1977}
J.~E. Fornaess and E.~L. Stout.
\newblock Spreading polydiscs on complex manifolds.
\newblock {\em Amer. J. Math.}, 99(5):933--960, 1977.

\bibitem{FornaessStout1982}
J.~E. Forn{\ae}ss and E.~L. Stout.
\newblock Regular holomorphic images of balls.
\newblock {\em Ann. Inst. Fourier (Grenoble)}, 32(2):v, 23--36, 1982.

\bibitem{FornaessWold2014}
J.~E. Forn{\ae}ss and E.~F. Wold.
\newblock Non-autonomous basins with uniform bounds are elliptic.
\newblock Proc. Amer. Math. Soc., to appear.

\bibitem{Forstneric2002MZ}
F.~Forstneri{\v{c}}.
\newblock The {O}ka principle for sections of subelliptic submersions.
\newblock {\em Math. Z.}, 241(3):527--551, 2002.

\bibitem{Forstneric2006AJM}
F.~Forstneri{\v{c}}.
\newblock Holomorphic flexibility properties of complex manifolds.
\newblock {\em Amer. J. Math.}, 128(1):239--270, 2006.

\bibitem{Forstneric2006AM}
F.~Forstneri{\v{c}}.
\newblock Runge approximation on convex sets implies the {O}ka property.
\newblock {\em Ann. of Math. (2)}, 163(2):689--707, 2006.

\bibitem{Forstneric2011}
F.~Forstneri{\v{c}}.
\newblock {\em Stein manifolds and holomorphic mappings}, volume~56 of {\em
  Ergebnisse der Mathematik und ihrer Grenzgebiete. 3. Folge. A Series of
  Modern Surveys in Mathematics [Results in Mathematics and Related Areas. 3rd
  Series. A Series of Modern Surveys in Mathematics]}.
\newblock Springer, Heidelberg, 2011.
\newblock The homotopy principle in complex analysis.

\bibitem{Forstneric2013}
F.~Forstneri{\v{c}}.
\newblock Oka manifolds: from {O}ka to {S}tein and back.
\newblock {\em Ann. Fac. Sci. Toulouse Math. (6)}, 22(4):747--809, 2013.
\newblock With an appendix by Finnur L{\'a}russon.

\bibitem{ForstnericLarusson2011}
F.~Forstneri{\v{c}} and F.~L{\'a}russon.
\newblock Survey of {O}ka theory.
\newblock {\em New York J. Math.}, 17A:11--38, 2011.

\bibitem{ForstnericLarusson2014}
F.~Forstneri{\v{c}} and F.~L{\'a}russon.
\newblock Holomorphic flexibility properties of compact complex surfaces.
\newblock {\em Int. Math. Res. Not. IMRN}, (13):3714--3734, 2014.

\bibitem{Gromov1989}
M.~Gromov.
\newblock Oka's principle for holomorphic sections of elliptic bundles.
\newblock {\em J. Amer. Math. Soc.}, 2(4):851--897, 1989.

\bibitem{GunningRossi}
R.~C. Gunning and H.~Rossi.
\newblock {\em Analytic functions of several complex variables}.
\newblock AMS Chelsea Publishing, Providence, RI, 2009.
\newblock Reprint of the 1965 original.

\bibitem{HormanderComplex}
L.~H{\"o}rmander.
\newblock {\em An introduction to complex analysis in several variables},
  volume~7 of {\em North-Holland Mathematical Library}.
\newblock North-Holland Publishing Co., Amsterdam, third edition, 1990.

\bibitem{KobayashiOchiai1975}
S.~Kobayashi and T.~Ochiai.
\newblock Meromorphic mappings onto compact complex spaces of general type.
\newblock {\em Invent. Math.}, 31(1):7--16, 1975.

\bibitem{Kodaira1971}
K.~Kodaira.
\newblock Holomorphic mappings of polydiscs into compact complex manifolds.
\newblock {\em J. Differential Geometry}, 6:33--46, 1971/72.

\bibitem{Kutzschebauch2014}
F.~Kutzschebauch.
\newblock Flexibility properties in complex analysis and affine algebraic
  geometry.
\newblock In {\em Automorphisms in birational and affine geometry}, volume~79
  of {\em Springer Proc. Math. Stat.}, pages 387--405. Springer, Cham, 2014.

\bibitem{LarussonTruong2016}
F.~L{\'a}russon and T.~T. Truong.
\newblock Algebraic subellipticity and dominability of blow-ups of affine
  spaces.
\newblock Preprint arXiv:1606.08115.

\bibitem{Low1983}
E.~L{\o}w.
\newblock An explicit holomorphic map of bounded domains in {${\bf C}^{n}$}\
  with {$C^{2}$}-boundary onto the polydisc.
\newblock {\em Manuscripta Math.}, 42(2-3):105--113, 1983.

\bibitem{Nomura2016}
R.~Nomura.
\newblock K{\"a}hler manifolds with negative holomorphic sectional curvature,
  k{\"a}hler-ricci flow approach.
\newblock Preprint arXiv:1610.01976.

\bibitem{TosattiYang2016}
V.~Tosatti and X.~Yang.
\newblock An extension of a theorem of wu-yau.
\newblock Preprint arXiv:1506.01145.

\bibitem{Winkelmann1993}
J.~Winkelmann.
\newblock The {O}ka-principle for mappings between {R}iemann surfaces.
\newblock {\em Enseign. Math. (2)}, 39(1-2):143--151, 1993.

\bibitem{Winkelmann2016}
J.~Winkelmann.
\newblock H-principle and specialness ii.
\newblock Lecture at the conference Frontiers in Elliptic Holomorphic Geometry,
  Jevnaker, Norway, 2016.

\bibitem{WuYau2016P}
D.~Wu and S.-T. Yau.
\newblock A remark on our paper "negative holomorphic curvature and positive
  canonical bundle".
\newblock Preprint arXiv:1609.01377.

\bibitem{WuYau2016IM}
D.~Wu and S.-T. Yau.
\newblock Negative holomorphic curvature and positive canonical bundle.
\newblock {\em Invent. Math.}, 204(2):595--604, 2016.

\end{thebibliography}


\vspace*{0.5cm}
\noindent Franc Forstneri\v c

\noindent Faculty of Mathematics and Physics, University of Ljubljana, Jadranska 19, SI--1000 Ljubljana, Slovenia

\noindent Institute of Mathematics, Physics and Mechanics, Jadranska 19, SI--1000 Ljubljana, Slovenia

\noindent e-mail: {\tt franc.forstneric@fmf.uni-lj.si}

\end{document}